\documentclass[a4paper, 11pt]{amsart} 
\usepackage[headings]{fullpage}               
\usepackage{boxedminipage}
\usepackage{amsfonts}
\usepackage{amsmath} 
\usepackage{amssymb}
\usepackage{graphicx}
\usepackage{amsthm}
\usepackage{subfig}
\usepackage{tikz}
\usetikzlibrary{cd}
\usepackage{setspace}
\usepackage[colorlinks=true,linkcolor=blue,urlcolor=blue,citecolor=red]{hyperref}
\usepackage[nameinlink]{cleveref}
\usepackage{enumerate}
\hypersetup{linktocpage=true,}
\newtheorem{theorem}{Theorem}[section]
\newtheorem{lemma}[theorem]{Lemma}
\newtheorem{conjecture}[theorem]{Conjecture}
\newtheorem{proposition}[theorem]{Proposition}
\newtheorem{corollary}[theorem]{Corollary}
\newtheorem{observation}[theorem]{Observation}

\theoremstyle{definition}\newtheorem{definition}[theorem]{Definition}
\theoremstyle{definition}\newtheorem{example}[theorem]{Example}
\theoremstyle{definition}\newtheorem{remark}[theorem]{Remark}

\newcommand{\rank}{\operatorname{rank}}

\newcommand{\spann}{\operatorname{span}}

\newcommand{\Iso}{\operatorname{Isom}}

\newcommand{\conv}{\operatorname{conv}}
\newcommand{\sgn}{\operatorname{sgn}}

\newcommand\scalemath[2]{\scalebox{#1}{\mbox{\ensuremath{\displaystyle #2}}}}

\definecolor{colR}{rgb}{.932,.172,.172}
\definecolor{colB}{rgb}{.255,.41,.884}
\definecolor{colG}{rgb}{0,0.7,0}

\tikzstyle{vertex}=[circle, draw, fill=black, inner sep=0pt, minimum size=4pt]
\tikzstyle{smallvertex}=[circle, line width=1.5pt, draw, fill=black, inner sep=0pt, minimum size=2pt]

\tikzstyle{edge}=[line width=1.5pt]
\tikzstyle{dedge}=[edge,dashed,gray]
\tikzstyle{redge}=[edge,colR]
\tikzstyle{bedge}=[edge,colB]
\tikzstyle{gedge}=[edge,colG]
\tikzstyle{lnode}=[circle,white,draw, fill=black,inner sep=1pt, font=\scriptsize]
\tikzstyle{hollow}=[circle,gray,draw, thick, fill=white,inner sep=0pt, minimum size=4pt]

\let\emph\relax % there's no \RedeclareTextFontCommand
\DeclareTextFontCommand{\emph}{\bfseries\em}

\begin{document}

\title{Infinitesimal rigidity and prestress stability for frameworks in normed spaces}

\author[Sean Dewar]{Sean Dewar*}\thanks{*Johann Radon Institute for Computational and Applied Mathematics, Austrian Academy of Sciences.\\Email: \nolinkurl{sean.dewar@ricam.oeaw.ac.at}\\The author was supported by the Austrian Science Fund (FWF): P31888.}

\begin{abstract}
	A (bar-and-joint) framework is a set of points in a normed space with a set of fixed distance constraints between them.
	Determining whether a framework is locally rigid -- i.e.~whether every other suitably close framework with the same distance constraints is an isometric copy -- is NP-hard when the normed space has dimension 2 or greater.
	We can reduce the complexity by instead considering derivatives of the constraints, 
	which linearises the problem.
	By applying methods from non-smooth analysis,
	we shall strengthen previous sufficient conditions for framework rigidity that utilise first-order derivatives.
	We shall also introduce the notions of prestress stability and second-order rigidity to the topic of normed space rigidity, two weaker sufficient conditions for framework rigidity previously only considered for Euclidean spaces.
% 	A simpler problem is determining whether a framework is infinitesimally rigid -- i.e.~ whether every infinitesimal constraint-preserving deformation corresponds to an infinitesimal isometry -- since the problem has been linearised.
% 	In smooth normed spaces, in particular Euclidean spaces, it can be shown via basic differential geometry methods that infinitesimal rigidity will imply local rigidity.
% 	However such methods fail when the norm, and hence the constraints of the framework, cannot be differentiated at all places in an open neighbourhood of the framework.
% 	Using the analogue of the constant rank theorem for Lipschitz maps,
% 	we introduce the property of strong infinitesimal rigidity, a property that is equivalent to infinitesimal rigidity for frameworks with differentiable distance constraints. With this we prove strong infinitesimal rigidity will always imply local rigidity for all frameworks,
% 	and thus infinitesimal rigidity will always imply local rigidity for frameworks with differentiable distance constraints.
\end{abstract}

\maketitle

\section{Introduction}

\subsection{Background}

Given a (finite dimensional real) normed space $X$ and a (finite simple) graph $G$,
we define a \emph{placement of $G$} to be a map $p :V \rightarrow X$,
and we call the pair $(G,p)$ a \emph{(bar-and-joint) framework in $X$}.
We define two frameworks $(G,p),(G,q)$ to be \emph{equivalent} if $\|p_v-p_w\|=\|q_v-q_w\|$ for all edges $vw$ of $G$,
and we define two placements $p,q$ to be \emph{congruent} if there exists an isometry $g:X \rightarrow X$ where $g \circ p = q$.
Given a framework $(G,p)$ in a normed space $X$,
we now wish to determine whether it is \emph{rigid} or \emph{flexible}, whatever those terms may mean.

Research into the rigidity and flexibility of frameworks in Euclidean spaces (finite-dimensional inner product spaces) can be seen to stem from the work of Cauchy \cite{cauchy}, Maxwell \cite{maxwell} and Kempe \cite{kempe}.
Research into rigidity and flexibility in other (finite-dimensional) normed spaces, however, is a much more recent endeavour.
The study of the rigidity of frameworks in normed spaces can be seen to date back to Kitson and Power \cite{kit-pow-1},
while research into flexible motions of frameworks in normed spaces was considered earlier by Cook, Lovett and Morgan \cite{clm}.
The closely related topic of graph flattenability -- determining whether all placements of a graph in an infinite-dimensional $\ell_p$ space can be embedded into a $d$-dimensional $\ell_p$ space whilst preserving edge lengths -- dates back even further to the work of Holsztynski \cite{hol78}, Witsenhausen \cite{wit86}, and Ball \cite{ball90},
and has been further continued in recent years by Willoughby and Sitharam \cite{SithWill} and Fiorini, Huynh, Joret, and Varvitsiotis \cite{fhjv2017}.

We have, as of yet, not defined what it means for a framework to be rigid.
At this point we have two reasonable definitions:
\begin{enumerate}[(i)]
    \item The framework $(G,p)$ is \emph{locally rigid} if there exists $\varepsilon >0$ so that every equivalent framework $(G,q)$ with $\|q_v-p_v\| < \varepsilon$ for every vertex $v$ is congruent;
    otherwise we say $(G,p)$ is \emph{locally flexible}.
    \item The framework $(G,p)$ is \emph{continuously rigid} if every \emph{continuous flex} $\gamma:[0,1]\rightarrow X^V$ (a continuous path with $\gamma(0)=p$ and $f_G(\gamma(t))=f_G(p)$ for all $t$) is \emph{trivial} ($\gamma(t)$ and $p$ are congruent for each value $t$);
    otherwise we say $(G,p)$ is \emph{continuously flexible}.
\end{enumerate}
Determining whether a framework has either of these rigidity conditions can be shown to be NP-hard in Euclidean spaces \cite{abbott},
and so will usually be equally difficult to determine in general normed spaces.
We can circumvent this issue by defining a stronger property as follows.

For a graph $G=(V,E)$ and normed space $X$,
let $f_G: X^V \rightarrow \mathbb{R}^E, ~ p \mapsto (\|p_v-p_w\|)_{vw \in E}$, be the \emph{rigidity map}.
An \emph{infinitesimal flex of $(G,p)$} is a vector $u \in X^V$ with 
\begin{align*}
    \lim_{t\rightarrow 0} \frac{1}{t}(f_G(p+tu) - f_G(p)) = \lim_{t\rightarrow 0} \left(\frac{\|p_v-p_w + t(u_v-u_w)\| - \|p_v-p_w\|}{t} \right)_{vw \in E} = (0)_{vw \in E}.
\end{align*}
Given $\Iso(X)$ is the Lie group of isometries of $X$ and $I:X \rightarrow X$ is the identity map for $X$,
every element of the linear space
\begin{align*}
    \mathcal{T}(p) : = \left\{ g \circ p : \text{ the map } g:X \rightarrow X \text{ lies in the tangent space of $\Iso(X)$ at $I$}  \right\}
\end{align*}
is an infinitesimal flex (see, for example, \cite{D21});
any such infinitesimal flex is called a \emph{trivial infinitesimal flex}.
We define a framework $(G,p)$ to be \emph{infinitesimally rigid} if every infinitesimal flex is trivial;
otherwise we say $(G,p)$ is \emph{infinitesimally flexible}.
If $(G,p)$ is \emph{well-positioned} -- i.e.~the rigidity map is (Fr\'{e}chet) differentiable at $p$ -- then this is equivalent to the (Fr\'{e}chet) derivative $df_G(p)$ of $f_G$ at $p$ (known as the \emph{rigidity operator} of $(G,p)$\footnote{In Euclidean spaces, the standard matrix representation of the rigidity operator for a well-positioned framework is exactly the \emph{rigidity matrix} (see, for example, \cite{gss}) after multiplying each row by $1/\|p_v-p_w\|$.}) having $\ker df_G(p) = \mathcal{T}(p)$.
We refer the reader to \cite{D21} for more details of the above definitions.

\subsection{Infinitesimal rigidity as a sufficient condition}\label{sec:intro2}

Asimow and Roth originally linked all three types of rigidity by showing that for a framework in a Euclidean space: (i) local and continuous rigidity are equivalent; (ii) all three types of rigidity are equivalent if the rigidity operator of the framework has maximal rank over all possible rigidity operators for placements of the same graph; and (iii) either almost all placements of a graph in a given Euclidean space are infinitesimally rigid, or none are \cite{asi-rot}.
It follows from (ii) that infinitesimal rigidity will imply local and continuous rigidity,
but the converse, however, is false; see \Cref{fig2}(i).

In \cite{D21},
the author proved the following analogues for all normed spaces:
(i) local rigidity implies continuous rigidity, but the converse statement is not true for all normed spaces; and (ii) all three types of rigidity are equivalent if the framework is \emph{constant} (i.e.~the derivative of the rigidity map will exist and its rank will remain constant on an open neighbourhood of the placement; see \cite[Section 4.2]{D21} for more details).
The third condition of Asimow and Roth was, however, shown by Kitson to be false in the case of polyhedral normed spaces (normed spaces where the unit ball forms a polytope); 
see \cite[Lemma 16]{kit}.
As the method implemented to prove point (ii) relied heavily on using differential geometry techniques to construct a differentiable manifold of equivalent placements,
it fairs badly in either of the following ``bad'' cases:
\begin{enumerate}
    \item\label{item1} the framework is \emph{badly-positioned}, i.e.~not well-positioned, or
    \item\label{item2} the framework is well-positioned but is not contained in a neighbourhood of well-positioned frameworks.
\end{enumerate}
(\ref{item1}) can be seen to be ``bad'' as badly-positioned infinitesimally rigid frameworks are not guaranteed to be locally/continuously rigid (see, for example, \Cref{fig2}(ii)).
(\ref{item2}), however, is only ``bad'' because the techniques used in differential geometry require maps that are continuously differentiable on open sets.
While we are always guaranteed that the set of well-positioned placements in a normed space will be dense (see \Cref{t:rad}),
we are not guaranteed that the set will always be open;
see, for example, \cite{govc} for a construction of an infinite family of normed spaces where every graph with at least one edge will have a dense set of badly-positioned placements.
In any such normed space we will not, with the current tools at our disposal, be able to utilise infinitesimal rigidity to find locally/continuously rigid frameworks.

\begin{figure}[tp]
	\begin{center}
        \begin{tikzpicture}[scale=2]
            \draw[->,red,thick] (1,0.5) -- (1.3,0.5);
            \draw[->,red,thick] (1,0.5) -- (0.7,0.5);
            
            \draw[->,red,thick] (0.5,1) -- (0.5,1.3);
            \draw[->,red,thick] (0.5,1) -- (0.5,0.7);
            
			\node[vertex] (1) at (0,0) {};
			\node[vertex] (2) at (1,0) {};
			\node[vertex] (2a) at (1,0.5) {};
			\node[vertex] (3) at (1,1) {};
			\node[vertex] (3a) at (0.5,1) {};
			\node[vertex] (4) at (0,1) {};
			
			\draw[edge] (1)edge(2);
			\draw[edge] (1)edge(3);
			\draw[edge] (2)edge(2a);
			\draw[edge] (2a)edge(3);
			\draw[edge] (2)edge(4);
			\draw[edge] (3)edge(3a);
			\draw[edge] (3a)edge(4);
			\draw[edge] (4)edge(1);
			
			\node at (0.5,-0.3) {(i)};
		\end{tikzpicture}
		\qquad\qquad\qquad\qquad
		\begin{tikzpicture}[scale=2]
			\draw[dashed,red,thick] (0.5,0.5) -- (-0.5,0.5) -- (-0.5,-0.5) -- (0.5,-0.5) -- (0.5,0.5);
			
			\node[vertex] (1) at (0,0) {};
			\node[vertex] (2) at (0.5,0.5) {};
			
			\draw[edge] (1)edge(2);
			\node at (0,-0.8) {(ii)};
		\end{tikzpicture}
	\end{center}
	\caption{(i) A framework in the Euclidean plane that is locally and continuously rigid but infinitesimally flexible; possible non-trivial infinitesimal flexes is indicated by the red arrows. (ii) A framework in the $\ell_\infty$ normed plane (see \Cref{sec:ex} for a definition) that is locally and continuously flexible but infinitesimally rigid; the dashed line represents other possible positions of the outer vertex.}\label{fig2}
\end{figure}
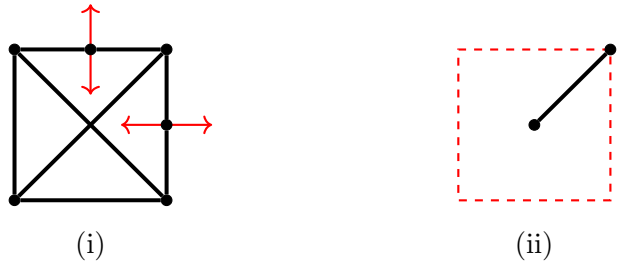

To solve the issues presented by (\ref{item1}) and (\ref{item2}),
we instead turn to non-smooth analytic methods;
i.e.~the study of convex and Lipschitz (see \Cref{def:lip}) maps and their properties.
We will apply these methods in \Cref{sec:main} to define \emph{strong infinitesimal rigidity} (\Cref{def:inf}).
Strong infinitesimal rigidity coincides with infinitesimal rigidity for well-positioned frameworks,
but has the added property that it will always imply local and continuous rigidity (\Cref{mainthm}),
hence fixing both (\ref{item1}) and (\ref{item2}).

\subsection{Prestress stability and second-order rigidity as sufficient conditions}

While determining whether a framework is infinitesimally rigid is relatively fast (it relies on determining the rank of a single matrix),
there are many locally/continuously rigid frameworks that are infinitesimally flexible;
for example, \Cref{fig2}(i).
One approach to solving the issue presented by the strictness of infinitesimal rigidity is to turn to second-order derivatives of the distance constraints.
Two properties of interest are \emph{prestress stability} and \emph{second-order rigidity} (see \Cref{def:prestress,def:2nd}), introduced by Connelly and Whiteley \cite{conn96} and Connelly \cite{conn80} respectively for Euclidean normed spaces.
Both of these conditions imply local/continuous rigidity in Euclidean normed spaces and can be determined quickly (although, unsurprisingly, not as quick as infinitesimal rigidity).

Unfortunately, many of the methods of Connelly and Whiteley fail to be applicable when we leave the relative safety of the Euclidean norm.
Indeed, it is not immediately obvious that there will even exist points where a non-Euclidean norm will be twice-differentiable,
although this was fortunately shown to be true in a very strong sense by Alexandrov's theorem (\Cref{t:alex}).
To get around many of these problems, we will once again turn to non-smooth analysis.
In \Cref{sec:2nd}, we will utilise non-smooth analysis results to extend prestress stability and second-order rigidity to general normed spaces.
We will prove that, like in Euclidean spaces, prestress stability will always imply both second-order and local rigidity,
although the author conjectures that second-order rigidity will not always imply local rigidity (\Cref{conj1}).

\subsection{Summary of results}\label{sec:summary}

We shall prove the following set of implications for framework rigidity in any normed space (the implication marked ``w.p.'' holds for well-positioned frameworks):
\begin{center}
\begin{tikzcd}
\text{Inf. rigid} \arrow[r, "\text{w.p.}", Rightarrow, bend right] & \text{Str. inf. rigid} \arrow[l, Rightarrow] \arrow[r, Rightarrow] & \text{Locally rigid} \arrow[r, Rightarrow] & \text{Cts. rigid}
\end{tikzcd}
\end{center}
We shall also prove that if the framework is also \emph{second-order well-positioned} (\Cref{def:2ndwp}),
we will have the following additional set of implications:
\begin{center}
\begin{tikzcd}
\text{(Str.) inf. rigid}   \arrow[r, Rightarrow] & \text{Pr. stable} \arrow[d, Rightarrow] \arrow[r, Rightarrow] & \text{Locally rigid} \\
                                                                      & \text{2nd-order rigid}                                        &                     
\end{tikzcd}
\end{center}
Towards the end of the paper, we will focus our results to particular normed spaces.
In \Cref{sec:norm}, we will strengthen our results for a variety of normed spaces.
We will follow this by showcasing some interesting examples in $\ell_p$ normed planes,
including the doubly-braced grids first investigated by Power in \cite{power20} (\Cref{sec:ex}). 
We shall conclude in \Cref{sec:end} with some final remarks and a conjecture regarding second-order rigidity.

\section{Non-smooth analysis}\label{sec:nonsmooth}

\subsection{Generalised derivatives}

For this section we will fix $X$ and $Y$ to be (finite-dimensional) normed spaces,
$D_1 \subset X$ and $D_2 \subset Y$ to be open sets and $f: D_1 \rightarrow D_2$ to be a continuous map.
We begin with the following definition.

\begin{definition}\label{def:lip}
    We say $f$ is \emph{Lipschitz} if there exists $K>0$ so that for every $x,y \in D_1$ we have $\|f(x)-f(y)\|_Y \leq K\|x-y\|_X$.
    We say $f$ is \emph{locally Lipschitz} if for every $x \in D_1$ there exists a neighbourhood $U$ so that the domain restricted map $f|_U$ is Lipschitz.
    If $f$ is a bijective Lipschitz map with a Lipschitz inverse map,
    then we say that $f$ is \emph{bilipschitz}.
\end{definition}

Every continuously differentiable map is locally Lipschitz, but the converse is not true.
We can say by a result of Rademacher that most points in the domain of a locally Lipschitz map will be differentiable.
For a proof of the following result, see \cite[Theorem 3.1.6]{federer}.

\begin{theorem}[Rademacher's theorem]\label{t:rad}
    Let $f: D_1 \rightarrow D_2$ be a locally Lipschitz map.
    Then the set of points in $D_1$ where $f$ is not differentiable has (Hausdorff) measure zero.
\end{theorem}

Using Rademacher's theorem,
we can define the following concept that was first introduced by Clarke \cite{clarke76}.
Denote the convex hull of a set $A$ by $\conv A$.
For a point $x \in D_1$,
define the set 
\begin{align*}
    \partial f(x) := \conv \left\{ \lim_{n \rightarrow \infty} df(x_n) : x_n \rightarrow x,~ f \text{ is differentiable at each } x_n \right\}.
\end{align*}
Any linear map $\delta f(x) \in \partial f(x)$ will be called a \emph{generalised derivative of $f$ at $x$}.
The generalised derivatives will have the following properties.

\begin{lemma}[see, for example, \cite{clarke}]\label{l:genderiv}
    Let $f: D_1 \rightarrow D_2$ be a locally Lipschitz map.
    Then the following properties hold:
    \begin{enumerate}[(i)]
        \item\label{l:genderiv1} $f$ is differentiable at a point $x$ if and only if $|\partial f(x)| = 1$.
        \item\label{l:genderiv2} The set $\partial f(x)$ will always be non-empty and compact.
        \item\label{l:genderiv3} For every point $x_0 \in D$ and every $\varepsilon >0$,
        there exists $\delta >0$ such that for all points $x \in D$ with $\|x-x_0\|_X<\delta$,
        we have 
        \begin{align*}
            \partial f(x) \subset \partial f(x_0) + B_\varepsilon,
        \end{align*}
        where $B_\varepsilon$ is the set of linear maps $T : X \rightarrow Y$ with $\|T(x)\|_Y < \varepsilon \|x\|_X$.
        Hence the derivative is continuous on the set of differentiable points.
    \end{enumerate}
\end{lemma}

Convex functions can be seen to be locally Lipschitz,
we can even tie their generalised derivatives into one-sided directional derivatives.

\begin{lemma}[see, for example, \cite{clarke}]\label{l:conv}
    Let $f: D_1 \rightarrow D_2$ be a convex function and $x,y \in D_1$.
    Then the one-way limit $\lim_{t \searrow 0} \frac{1}{t}(f(x+ty)-f(x))$ exists and is equal to $\max_{F \in \partial f(x)} F(u)$.
\end{lemma}

With this generalisation of differentiation for locally Lipschitz maps,
Butler, Timourian and Viger were able to give a Lipschitzian analogue of the constant rank theorem.

\begin{theorem}[\cite{btv}]\label{t:btv}
    Let $f: D_1 \rightarrow D_2$ be a locally Lipschitz map, and let $m = \dim X$ and $n = \dim Y$.
    Suppose that every generalised derivative of $f$ at every point $x \in D_1$ has rank $k$.
    Then for every point $z \in D_1$, there exists open sets $U \subset \mathbb{R}^m$, $V \subset \mathbb{R}^n$, $U' \subset D_1$, $V' \subset D_2$ with $z \in U'$ and $f(z) \in V'$, and bilipschitz maps $\phi : U \rightarrow U'$ and $\psi:V \rightarrow V'$,
    so that for every point $x=(x_1,\ldots,x_n) \in U$ we have
    \begin{align*}
        \psi \circ f \circ \phi^{-1} (x_1,\ldots,x_m) = (x_1,\ldots, x_k).
    \end{align*}
\end{theorem}

\subsection{Hessians}

For this section we will fix $X$ to be a normed space and $f: X \rightarrow \mathbb{R}$ to be a continuous function.

\begin{definition}
    Suppose that the function $f$ is differentiable at a point $x\in X$.
    We say that $f$ is \emph{twice-differentiable at $x$} if there exists a symmetric bilinear function $d^2f(x) : X \times X \rightarrow \mathbb{R}$ so that
    \begin{align*}
        \lim_{\|h\| \rightarrow 0} \frac{f(x+h) - f(x) - df(x)(h) - \frac{1}{2}d^2f(x)(h,h)}{\|h\|^2} = 0;
    \end{align*}
    we call $d^2f(x)$ (if it exists) the \emph{Hessian of $f$ at $x$}.
\end{definition}

It can be quickly checked that if a function has a Hessian at a point,
then it must also be unique.
If $f$ and $df$ are both continuously differentiable on an open set $U \subset X$,
then $d^2f(x)$ will be the derivative of $df$ at each point $x \in U$.
The Hessian of a function at a point can be very useful,
especially if we wish to determine whether a point is a local minimum or not.
We will first need the following result regarding isometries of normed spaces.

\begin{lemma}[{\cite[Lemma 3.4]{D21}}]\label{l:isom}
    Let $\mathcal{G}$ be a subgroup of the isometries of $X$ with tangent space $T_I \mathcal{G}$ at the identity map $I$.
    Then for each $x \in X$,
    the set $\mathcal{G} \cdot x := \{ g(x): g \in \mathcal{G}\}$ is a closed\footnote{Here we mean closed in the topological sense. We are not assuming that the manifold is also compact.} smooth manifold with tangent space $T_I \mathcal{G} \cdot x := \{ g(x): g \in T_I \mathcal{G}\}$ at $x$.
\end{lemma}

\begin{lemma}\label{l:minimum}
    Let $\mathcal{G}$ be a subgroup of the group of isometries of $X$.
    Suppose $f(g.x) = f(x)$ for all $g \in \mathcal{G}$ and $x\in X$.
    If:
    \begin{enumerate}[(i)]
        \item $f$ is twice-differentiable at $x_0$,
        \item the derivative $df(x_0)$ is the zero map, i.e.~$df(x_0)(u) = 0$ for all $u \in X$,
        \item the Hessian $d^2f(x_0)$ is \emph{positive semi-definite},
        i.e.~$d^2f(x_0)(u,u) \geq 0$ for all $u \in X$, and
        \item for any $u \in X$ we have $d^2f(x_0)(u,u) =0$ if and only if $u \in T_I \mathcal{G} \cdot x_0$;
    \end{enumerate}
    then $x_0$ is a local minimum of $f$, and there exists a neighbourhood $U$ of $x_0$ where $f^{-1}[f(x_0)]\cap U = (\mathcal{G}\cdot x) \cap U$.
\end{lemma}

\begin{proof}
    Since $X$ can be renormed by an equivalent inner product norm so that $\mathcal{G}$ remains a subgroup of the isometries (see, for example, \cite[Cor.~3.3.4]{minkowski}),
    we may assume that $X$ is an inner product space with inner product $\langle \cdot, \cdot \rangle$.
    Define $H \subset X$ to be the orthogonal complement of $T_I \mathcal{G} \cdot x_0$,
    i.e.~$\langle h, z \rangle = 0$ for all $h \in H$ and $z \in T_I \mathcal{G} \cdot x_0$.
    If we set 
    \begin{align*}
        c:= \inf \left\{ d^2 f(x_0)(h,h) : h \in H, ~ \|h\|=1\right\},
    \end{align*}
    then as $d^2f(x_0)$ is positive definite on $H$ 
    (i.e.~$d^2f(x_0)(u,u) > 0$ for all non-zero $u \in H$),
    we observe that $c >0$.
    Since Hessian of $f$ is positive definite on $H$ and $df(x_0)$ is the zero map,
    we have from rearranging the definition of the Hessian the following;
    for every $\varepsilon >0$,
    there exists $\delta >0$ such that if $\|h\|<\delta$ and $h \in H$ then
    \begin{align*}
        \frac{f(x_0+h) - f(x_0)}{\|h\|^2} \ > \ \frac{1}{2} d^2 f(x_0) \left(\frac{h}{\|h\|}, \frac{h}{\|h\|}  \right) - \varepsilon \ \geq \ c-\varepsilon.
    \end{align*}
    Hence there exists $\delta' >0$ so that if $h \in H \setminus \{0\}$ with $\|h\|<\delta'$,
    then $f(x_0) > f(x_0)$.
    
    Define the open set $O \subset X$ of points that are less than $\delta'$ distance from $\mathcal{G} \cdot x_0$,
    and fix a point $x \in (O \setminus \mathcal{G} \cdot x_0)$.
    Since $\mathcal{G} \cdot x_0$ is a closed subset of $X$,
    there exists $g \in \mathcal{G}$ so that $g (x_0)$ is the closest point to $x$ in $\mathcal{G} \cdot x_0$.
    The vector $x- g(x_0)$ must be perpendicular to the tangent space of $\mathcal{G} \cdot x_0$ at $g(x_0)$,
    as otherwise there would exist a point in $\mathcal{G} \cdot x_0$ closer to $x$.
    Since the tangent space of $\mathcal{G} \cdot x_0$ at $g(x_0)$ is $g(T_I \mathcal{G} \cdot x_0)-g(0)$,
    we have $x = g(x_0 + h)$ for some $h \in H \setminus \{0\}$ with $\|h\|<\delta'$.
    Hence $f(x) = f(g(x_0 + h)) = f(x_0 +h) > f(x_0)$ as required.
\end{proof}

A convex functions will be twice-differentiable at almost all points by the following famous result of Alexandrov \cite{alex}.

\begin{theorem}[Alexandrov's theorem]\label{t:alex}
    If $f: X \rightarrow \mathbb{R}$ is a convex function,
    then there exists a conull\footnote{The complement of the set has Hausdorff measure zero.} subset $D' \subset X$ of twice-differentiable points of $f$.
\end{theorem}

\section{Strong infinitesimal rigidity}\label{sec:main}

Previously we could only define the rigidity operator of well-positioned frameworks.
Since the rigidity map is Lipschitz,
we define any generalised derivative $\delta f_G(p)$ of $f_G$ at $p$ to be a \emph{generalised rigidity operator} of $(G,p)$.
As the space $\mathcal{T}(p)$ will lie in the kernel of all the rigidity operator limits that generate $\partial f_G(p)$,
it is immediate that $\mathcal{T}(p)$ will lie in the kernel of all the generalised rigidity operators of $(G,p)$.
By \Cref{l:genderiv}(\ref{l:genderiv1}),
a framework will be well-positioned if and only if it has a single generalised rigidity operator.
We can use generalised rigidity operators to give an alternative definition of infinitesimal rigidity.

\begin{proposition}\label{p:altdef}
    The set of infinitesimal flexes of a framework $(G,p)$ in a normed space $X$ is exactly the linear space
    \begin{align*}
        \mathcal{F}(G,p) := \bigcap_{\delta f_G(p) \in \partial f_G(p)} \ker \delta f_G(p).
    \end{align*}
    Hence $(G,p)$ is infinitesimally rigid if and only if $\mathcal{F}(G,p) = \mathcal{T}(p)$.
\end{proposition}

\begin{proof}
    For each $u \in X^V$, define the one-way limit
    \begin{align*}
        f_G'(p;u) := \lim_{t \searrow 0} \frac{1}{t}(f_G(p+tu)-f_G(p)).
    \end{align*}
    Since each component function of $f_G$ is convex,
    it follows from \Cref{l:conv} that both $f_G'(p;u)$ and $f_G'(p;-u)$ exist,
    and $f'_G(p;u)=f'_G(p;-u)=0$ if and only if $\delta f_G(p)(u) = 0$ for all $\delta f_G(p) \in \partial f_G(p)$ as required.
\end{proof}

We will require later on that our set of equivalent frameworks will a local manifold structure at a given placement.
To facilitate this, we define the following.

\begin{definition}\label{def:constant}
    A framework $(G,p)$ in a normed space $X$ is \emph{weakly constant} if there exists a non-negative integer $k$ and an open neighbourhood $U \subset X^V$ of $p$ where for all $q \in U$ and $\delta f_G(q) \in \partial f_G(q)$, we have $\rank \delta f_G(q) = k$.
\end{definition}

A weakly constant framework will be constant (see \Cref{sec:intro2}) if and only if it has a neighbourhood of well-positioned frameworks;
hence every constant framework will also be weakly constant.
Since the rank function is lower semi-continuous,
any framework with only surjective generalised rigidity operators will be weakly constant by \Cref{l:genderiv}(\ref{l:genderiv3}).

The next result demonstrates why weakly constant is an important property for a framework.
We recall that a \emph{Lipschitz manifold} is a topological manifold where the transition maps are all locally Lipschitz.
We will also denote by $\mathcal{O}_p$ the set of placements of a graph $G=(V,E)$ in a normed space $X$ congruent to a placement $p \in X^V$.
In \cite{D21} it was shown $\mathcal{O}_p$ is a smooth manifold with tangent space $\mathcal{T}(p)$ at $p$.

\begin{lemma}\label{l:manifold}
    Let $(G,p)$ be a weakly constant framework in a normed space $X$.
    Then there exists an open neighbourhood $U$ of $p$ where the set $f_G^{-1}[f_G(p)] \cap U$ is a connected Lipschitz manifold of dimension $m$, where $m$ is the nullity of any of the generalised rigidity operators of $(G,p)$.
    Furthermore,
    the set $\mathcal{O}_p \cap U$ is a connected Lipschitz submanifold of $f_G^{-1}[f_G(p)] \cap U$ with dimension $\dim \mathcal{T}(p)$.
\end{lemma}

\begin{proof}
    By \Cref{t:btv},
    there exists an open neighbourhood $U$ of $p$ where the set $f_G^{-1}[f_G(p)] \cap U$ is a Lipschitz submanifold of $X^V$ of dimension $m$.
    Since $\mathcal{O}_p \cap U$ is a Lipschitz submanifold of $X^V$ also,
    then $\mathcal{O}_p \cap U$ is a Lipschitz submanifold of $f_G^{-1}[f_G(p)] \cap U$ of dimension $\dim \mathcal{T}(p)$.
    Since Lipschitz manifolds are locally path-connected,
    we may suppose that $U$ is sufficiently small that both $f_G^{-1}[f_G(p)] \cap U$ and $\mathcal{O}_p \cap U$ are connected.
\end{proof}

We can now extend our equivalence of rigidity properties for constant frameworks to weakly constant frameworks.

\begin{lemma}\label{mainlem}
    Let $(G,p)$ be a weakly constant framework in a normed space $X$.
    Then the following are equivalent:
    \begin{enumerate}[(i)]
        \item\label{mainthm2item1} There exists a generalised rigidity operator $\delta f_G(p)$ of $(G,p)$ with $\ker \delta f_G(p) = \mathcal{T}(p)$.
        \item\label{mainthm2item2} $(G,p)$ is continuously rigid.
        \item\label{mainthm2item3} $(G,p)$ is locally rigid.
    \end{enumerate}
\end{lemma}

\begin{proof}
    By \Cref{l:manifold},
    there exists an open neighbourhood $U$ of $p$ where $f_G^{-1}[f_G(p)] \cap U = \mathcal{O}_p \cap U$ if and only if $\dim \ker \delta f_G(p) = \dim \mathcal{T}(p)$ for some $\delta f_G(p) \in \partial f_G(p)$.
    Hence (\ref{mainthm2item1}) and (\ref{mainthm2item2}) are equivalent.
    Since Lipschitz manifolds are locally path-connected,
    the rest of the proof is identical to \cite[Theorem 1.1]{D21}.
\end{proof}

Although \Cref{mainlem} is useful,
having to check the rank of every generalised rigidity operator of every framework in some neighbourhood will not be practical.
Checking the rank of all the generalised rigidity operators of a single framework, however, is a much more tractable problem due to the convexity of $\partial f_G(p)$.
With this in mind, 
we define the following.

\begin{definition}\label{def:inf}
    Let $(G,p)$ be a framework in a normed space $X$.
    We define $(G,p)$ to be \emph{strongly infinitesimally rigid} if for every generalised rigidity operator $\delta f_G(p)$ of $(G,p)$ we have $\ker \delta f_G(p) = \mathcal{T}(p)$;
    otherwise we say $(G,p)$ is \emph{strongly infinitesimally flexible}.
\end{definition}

We first make the following observations about our new property.
The first follows, in part, from \Cref{l:genderiv}(\ref{l:genderiv1}), while the latter observation can be seen by checking the rank of any generalised rigidity operator.

\begin{observation}
    Every strongly infinitesimally rigid framework is infinitesimally rigid,
    and every well-positioned infinitesimally rigid framework is strongly infinitesimally rigid also.
\end{observation}

(There do, however, exist infinitesimally rigid frameworks that are not strongly infinitesimally rigid;
for instance, the framework defined in \Cref{fig2}(ii).)

\begin{observation}
    If $(G,p)$ is a strongly infinitesimally rigid framework in a normed space $X$,
    then, given $G=(V,E)$, we have $|E| \geq (\dim X) |V| - \dim \mathcal{T}(p)$.
\end{observation}

Importantly, all strongly infinitesimally rigid frameworks are also weakly constant.

\begin{lemma}\label{l:infimplconst}
    If a framework is strongly infinitesimally rigid,
    then it is weakly constant.
    Furthermore,
    every framework in a sufficiently small neighbourhood will also be strongly infinitesimally rigid.
\end{lemma}

\begin{proof}
    Let $(G,p)$ be a strongly infinitesimally rigid framework in a normed space $X$,
    i.e.~the rank of every generalised rigidity operator of $(G,p)$ is $(\dim X)|V|-\dim \mathcal{T}(p)$.
    Then by \Cref{l:genderiv}(\ref{l:genderiv3}) and the lower semi-continuity of the rank function,
    there exists an open neighbourhood $U'$ of $p$ where for all $q \in U'$ and $\delta f_G(q) \in \partial f_G(q)$ we have
    \begin{equation}\label{eq1}
        (\dim X)|V| - \dim \mathcal{T}(p) \leq \rank \delta f_G(q) \leq (\dim X)|V| - \dim \mathcal{T}(q).
    \end{equation}
    If $|V| \leq \dim X$ then by \cite[Prop. 5.7 \& Thm 5.8]{D21},
    we have that $X$ is Euclidean and $(G,p)$ is independent,
    so we may suppose $|V| \geq \dim X +1$.
    By \cite[Prop. 3.13 \& Lemma 4.9(ii)]{D21},
    there exists an open set $U \subset U'$ where $p \in U$ and $\dim \mathcal{T}(q) = \dim \mathcal{T}(p)$ for all $q \in U$.
    Combining this with \cref{eq1},
    we have that $(G,p)$ is constant and every framework in $U$ is also infinitesimally rigid.
\end{proof}

We are now ready to prove our main result of the section.

\begin{theorem}\label{mainthm}
    Let $X$ be a normed space. 
    Then every strongly infinitesimally rigid framework in $X$ is locally rigid,
    and every locally rigid framework in $X$ is continuously rigid.
\end{theorem}

\begin{proof}
    Every strongly infinitesimally rigid framework is locally rigid by \Cref{l:infimplconst,mainlem}.
    That every locally rigid framework is continuously rigid can be seen from the proof of \cite[Theorem 1.1]{D21}.
\end{proof}

Combining \Cref{l:infimplconst} with \Cref{mainthm},
we obtain the immediate corollary.

\begin{corollary}\label{maincor}
    Let $X$ be a normed space.
    If there exists a well-positioned infinitesimally rigid framework $(G,p)$ in $X$,
    then there exists $\varepsilon > 0$ such that if $(G,q)$ is a framework with $\|p_v-q_v\|<\varepsilon$ for all $v \in V$,
    then $(G,q)$ is (strongly) infinitesimally, locally and continuously rigid.
\end{corollary}

\begin{remark}
    Neither of the converse statements given in \Cref{mainthm} will hold in general.
    An example of a locally rigid but (strongly) infinitesimally flexible framework in the Euclidean plane can be seen in \Cref{fig2}(i).
    For an example of a continuously rigid but locally flexible framework, we refer the reader to \cite{D21}.
\end{remark}

\section{Types of rigidity requiring the Hessian of the norm}\label{sec:2nd}

\subsection{Prestress stability}
Since infinitesimal rigidity is a strictly stronger property than local/continuous rigidity in Euclidean spaces,
one may start to wonder if there is a type of rigidity that lies in between.
One idea first put forward by Connelly and Whiteley in \cite{conn96} was to observe properties of the second-order derivative of the rigidity map.
Unlike with Euclidean spaces, we have to be slightly careful here, since the distance constraints of our frameworks are not guaranteed to be twice-differentiable. 

\begin{definition}\label{def:2ndwp}
    A framework $(G,p)$ in a normed space $X$ is \emph{second-order well-positioned} if the norm is twice-differentiable at each point $p_v-p_w$ for every $vw \in E$.
\end{definition}

The next result now follows immediately from \Cref{t:rad,t:alex}.

\begin{corollary}
    The set of (second-order) well-positioned placements of a graph $G=(V,E)$ in a normed space $X$ is a conull subset of $X^V$.
\end{corollary}

We will now define $\varphi_x$ and $\Delta_x$ to be the derivative and the Hessian (if they exist) of the point $x$ in a normed space $X$.
With this, we are now ready for a new type of rigidity.

\begin{definition}\label{def:prestress}
    Let $(G,p)$ be a second-order well-positioned framework in a normed space $X$.
    We say that $(G,p)$ is \emph{prestress stable} if there exists $a, b \in \mathbb{R}^E$ so that: (i) $b_{vw} >0$ for each $vw \in E$, (ii) $a$ is a \emph{(normalised equilibrium) stress of $(G,p)$}, i.e.~$\sum_{w \in N(v)} a_{vw} \varphi_{p_v-p_w} =0$ for each $v \in V$ with neighbourhood $N(v)$, and (iii) the quadratic form 
    \begin{align*}
        H_{a,b} : X^V \rightarrow \mathbb{R},~ u \mapsto \sum_{vw \in E} a_{vw} \Delta_{p_v-p_w}(u_v-u_w,u_v-u_w) + b_{vw} \left( \varphi_{p_v-p_w}(u_v-u_w) \right)^2
    \end{align*}
    is non-negative with zero set $\mathcal{T}(p)$.
\end{definition}

\begin{remark}
    Readers more familiar to rigidity theory in Euclidean space will notice that our definition for prestress stability differs from that given by \cite{conn96}.
    Our definition is, however, equivalent to Connelly and Whiteley's in Euclidean spaces (see \Cref{c:equiv defs}), since we have opted to use the norm instead of the norm squared for our rigidity map.
\end{remark}

The usefulness of prestress stability stems from the following result.

\begin{theorem}\label{main2nd}
    Let $(G,p)$ be a second-order well-positioned framework in a normed space $X$.
    \begin{enumerate}[(i)]
        \item If $(G,p)$ is infinitesimally rigid, then it is prestress stable.
        \item If $(G,p)$ is prestress stable, then it is locally rigid.
    \end{enumerate}
\end{theorem}

\begin{proof}
    First suppose $(G,p)$ is infinitesimally rigid.
    Choose the stress $a=(0)_{vw \in E}$ of $(G,p)$ and the vector $b = (1)_{vw \in E}$.
    We now note that $H_{a,b}(u) = \|df_G(p)(u)\|^2_{2}$, where $\|\cdot\|_{2}$ is the Euclidean norm on $\mathbb{R}^E$.
    Hence $H_{a,b}$ is non-negative with zero set $\mathcal{T}(p)$.
    
    Now suppose $(G,p)$ is prestress stable with respect to $a,b$.
    For each $vw$,
    let $f_{vw} :\mathbb{R} \rightarrow \mathbb{R}$ be the quadratic function where $f_{vw}'(\|p_v-p_w\|) =a_{vw}$ and $f_{vw}''(\|p_v-p_w\|) =b_{vw}>0$.
    Define the function $F:\mathbb{R}^E \rightarrow \mathbb{R}$ where $F(x) = (f_{vw}(x_{vw}))_{vw\in E}$.
    We now note that $F \circ f_G$ is a convex function that is twice-differentiable at $p$, and will have derivative $df_G(p)^T(a)=0$ and a Hessian with corresponding quadratic form $H_{a,b}$.
    By applying \Cref{l:minimum} to the function $F \circ f_G$ and normed space $X^V$ with norm $\sum_{v\in V} \|x_v\|$,
    there exists a neighbourhood $U$ of $p$ where for each $q \in U$ we have $F\circ f_G(q) = F\circ f_G(p)$ if and only if $q \in \mathcal{O}_p$.
    Since $F\circ f_G(q) = F\circ f_G(p)$ for all equivalent frameworks $(G,q)$,
    we have that $(G,p)$ is locally rigid.
\end{proof}

\begin{remark}\label{rem1}
    It can be seen from the proof of \Cref{main2nd} that for any stress $a\in \mathbb{R}^E$ of $(G,p)$ and any vector $b \in \mathbb{R}^E$ with positive coefficients,
    the set $\mathcal{T}(p)$ must lie in the zero set of $H_{a,b}$.
    To see this, we observe that if we restrict the corresponding convex function $F \circ f_G$ to the smooth manifold $\mathcal{O}_p$ of congruent placements, then it is everywhere zero.
    This justifies our definition of prestress stable requiring that the zero set of $H_{a,b}$ is exactly $\mathcal{T}(p)$ instead of just requiring that the zero set is contained in $\mathcal{T}(p)$.
\end{remark}

Although \Cref{def:prestress} gives a better understanding of how prestress stability works,
the more practical formulation is as follows.

\begin{proposition}\label{p:prestress}
    Let $(G,p)$ be a second-order well-positioned framework in a normed space $X$.
    Then $(G,p)$ is prestress stable if and only if there exists a stress $a \in \mathbb{R}^E$ of $(G,p)$ so that
    \begin{align*}
        H_a(u) := \sum_{vw \in E} a_{vw} \Delta_{p_v-p_w}(u_v-u_w,u_v-u_w) >0
    \end{align*}
    for all $u \in \ker df_G(p) \setminus \mathcal{T}(p)$.
\end{proposition}

\begin{proof}
    Fix $a,b \in \mathbb{R}^E$ with $a$ being a stress of $(G,p)$ and $b_{vw} >0$ for all $vw \in E$. Define the non-negative quadratic $H_b := H_{a,b}-H_a$.
    It is immediate that $H_b(u) =0$ if and only if $u \in \ker df_G(p)$.
    Suppose first that $(G,p)$ is prestress stable with respect to $a,b \in \mathbb{R}^E$.
    Then we have that $H_{a}(u) = H_{a,b}(u) >0$ for all $u \in \ker df_G(p) \setminus \mathcal{T}(p)$.
    Now suppose that $H_a(u)(u) > 0$ for all $u \in \ker df_G(p) \setminus \mathcal{T}(p)$.
    Since $H_b=0$ on $\ker df_G(p)$,
    it follows from the observation made in \Cref{rem1} that $H_a =0$ on $\mathcal{T}(p)$.
    Since $H_a > 0$ on $\ker df_G(p) \setminus \mathcal{T}(p)$ and $H_b$ is non-negative and positive on $X^V \setminus \ker df_G(p)$,
    we can choose sufficiently large $\lambda >0$ so that $H_{\lambda b}(u) = \lambda H_b(u) > H_a(u)$ for all $u \notin \ker df_G(p)$ (see \cite[Lemma 3.4.1]{conn96}).
    It follows that $H_{a, \lambda b}(u) >0$ for all $u \notin \mathcal{T}(p)$ as required.
\end{proof}

\subsection{Second-order rigidity}
Connelly and Whiteley also described another weaker type of rigidity in \cite{conn96} which we will now define for all normed spaces.

\begin{definition}\label{def:2nd}
    Let $(G,p)$ be a second-order well-positioned framework in a normed space $X$.
    A pair $(u,u') \in X^V \times X^V$ is said to be a \emph{second-order (infinitesimal) flex} of $(G,p)$ if $u$ is an infinitesimal flex of $(G,p)$ and
    \begin{align*}
        \Delta_{p_v-p_w}(u_v-u_w,u_v-u_w) + \varphi_{p_v-p_w}(u_v'-u_w') = 0
    \end{align*}
    for every edge $vw \in E$.
    If every second-order flex of $(G,p)$ has the property that $u \in \mathcal{T}(p)$,
    then we say $(G,p)$ is \emph{second-order (infinitesimal) rigid}; otherwise we say $(G,p)$ is \emph{second-order (infinitesimal) flexible}.
\end{definition}

It is immediate from the definition that for second-order well-positioned frameworks, infinitesimal rigidity will always imply second-order rigidity.
We can also easily show the following implication.

\begin{proposition}\label{main prestress implies 2nd}
    Let $(G,p)$ be a second-order well-positioned framework in a normed space $X$.
    If $(G,p)$ is prestress stable, then it is second-order rigid.
\end{proposition}

\begin{proof}
    Suppose $(G,p)$ has a second-order flex $(u,u')$ with $u \in \ker df_G(p) \setminus \mathcal{T}(p)$.
    We note that for any stress $a$ of $(G,p)$, we have
    \begin{eqnarray*}
        H_a(u) &=& \sum_{vw \in E} a_{vw} \Delta_{p_v-p_w}(u_v-u_w,u_v-u_w)  + \frac{1}{2} \sum_{v \in V}\sum_{w \in N(v)} a_{vw} \varphi_{p_v-p_w} (u'_v-u'_w) \\
        &=& \sum_{vw \in E} a_{vw} \left( \Delta_{p_v-p_w}(u_v-u_w,u_v-u_w) + \varphi_{p_v-p_w} (u'_v-u'_w) \right) \ = \ 0.
    \end{eqnarray*}
    Hence $(G,p)$ is not prestress stable by \Cref{p:prestress}.
\end{proof}

It is unclear whether second-order rigidity implies local or continuous rigidity in all normed spaces,
though the author would conjecture that this does not always hold (see \Cref{conj1}).

\section{Results for various classes of normed spaces}\label{sec:norm}

We shall now focus on two well-known specific classes of normed spaces.

\subsection{Polyhedral normed spaces}
We begin with \emph{polyhedral normed spaces};
i.e.~normed spaces $X$ where there exists a finite set of linear functionals $f_1,\dots,f_n:X \rightarrow \mathbb{R}$ so that $\|x\| = \max_{i=1,\ldots,n} |f_i(x)|$ for all $x \in X$.
Rigidity in such normed spaces was originally investigated by Kitson in \cite{kit}.

\begin{proposition}\label{p:poly}
    Let $(G,p)$ be a framework in a polyhedral normed space $X$.
    Then the following properties hold.
    \begin{enumerate}[(i)]
        \item\label{p:poly1} $(G,p)$ is locally rigid if and only if it is continuously rigid.
        \item\label{p:poly2} If $(G,p)$ is well-positioned, then it is second-order well-positioned.
        \item\label{p:poly3} If $(G,p)$ is well-positioned,
        then all types of rigidity (i.e.~(strong) infinitesimal rigidity, prestress stability, second-order rigidity, local rigidity, continuous rigidity) are all equivalent.
    \end{enumerate}
\end{proposition}

\begin{proof}
    We may suppose that the norm of $X$ is generated by the finite set of linear functionals $S$ with $-S = S$.
    Order the vertices of $G$.
    Define $S^E$ to be the set of all maps from $E$ to $S$,
    and define for each map $h:E \rightarrow S$ in $S^E$ the linear map 
    \begin{align*}
            T_h :X^V \rightarrow \mathbb{R}^E,~ x \mapsto ( h_{vw}(x_v-x_w))_{vw \in E, ~ v > w}.
    \end{align*}
    Given $\preceq$ is the partial ordering on $\mathbb{R}^E$ given by the equivalence $a \leq b$ if and only if $a_{vw} \leq b_{vw}$ for all $vw \in E$,
    define for each $h \in S^E$ the (possibly empty) set 
    \begin{align*}
        C_h := \left\{ q \in X^V : T_h(q) \succeq T_{g}(q) \text{ for all } g \in S^E \right\}.
    \end{align*}
    Each set $C_h$ will have the following properties:
    (i) each $C_h$ will be either empty or a closed convex cone, 
    (ii) every point $q \in X^V$ lies in some $C_h$, and
    (iii) the restriction of the rigidity map to every $C_h$ is the linear map $T_h$.
    We now fix the set $\Phi(p) := \{ h \in E^S : p \in C_h\}$ and note that $\bigcup_{h \in \Phi(p)} C_h$ is a neighbourhood of $p$.
    
    If $(G,p)$ is locally rigid then it will be continuously rigid by \Cref{mainthm}.
    Suppose $(G,p)$ is locally flexible,
    i.e.~there exists a sequence of placement $(p^n)_{n \in \mathbb{N}}$ that converge to $p$ where each framework $(G,p^n)$ is equivalent but non-congruent to $(G,p)$.
    As $\bigcup_{h \in \Phi(p)} C_h$ is a neighbourhood of $p$,
    we must have for some sufficiently large $n$ that $p^n \in C_h$ for some $h \in \Phi(p)$.
    Since $f_G$ is linear on $C_h$,
    there exists a continuous path $\gamma : t \mapsto (1-t) p + t p^n$, $t \in [0,1]$, where $f_G(\gamma(t)) = f_G(p)$ for each $t \in [0,1]$ as required.
    
    Now suppose $(G,p)$ is well-positioned.
    We note that this is equivalent to $\Phi(p) = \{h\}$ for some $h$.
    Since $f_G$ is linear on $C_h$, which in turn is a neighbourhood of $p$,
    it follows that all types of rigidity for $(G,p)$ are equivalent as required.
\end{proof}

It follows from \Cref{p:poly} that prestress stability and second-order rigidity will never be useful things to consider when dealing with polyhedral normed spaces.
We further note that the conditions for prestress stability and second-order rigidity will always simplify to being exactly the condition for infinitesimal rigidity,
since the Hessian of any twice-differentiable point of a polyhedral norm will always be the zero matrix.

\subsection{\texorpdfstring{$\ell_p$}{Lp} normed spaces}
Our next class of normed space are the \emph{$\ell_p$ normed spaces}, denoted by $\ell_p^d = (\mathbb{R}^d,\|\cdot\|_p)$ for some $p \in [1, \infty]$\footnote{Whenever we are dealing with an $\ell_p$ space, we shall reserve $p$ to be a scalar and shall instead use $q$ to represent a placement of a graph.},
where for each $x=(x_1,\ldots,x_d) \in \mathbb{R}^d$ we have
\begin{align*}
    \|x\|_p := \left( \sum_{i=1}^d |x_i|^p \right)^{1/p} \text{ if $p \neq \infty$,} \qquad \|x\|_\infty := \max_{i=1,\ldots,d} |x_i|.
\end{align*}
Both $\ell_1^d$ and $\ell_\infty^d$ are polyhedral normed spaces,
so we shall focus on those $\ell_p$ normed spaces where $p \in (1,\infty)$.
Given $\sgn: \mathbb{R} \rightarrow \{-1,0,1\}$ is the sign function (i.e.~$\sgn(a) = 1$ if $a>0$, $\sgn(a) = -1$ if $a<0$, and $\sgn(a)=0$ if $a=0$) and $p \in (1,\infty)$, the norm $\|\cdot\|_p$ is a strictly convex function that is differentiable at every point $x\neq 0$ with gradient
\begin{align*}
    \frac{1}{\|x\|_p^{p-1}} \left( \sgn(x_1)|x_1|^{p-1}, ~ \ldots ~ , ~  \sgn(x_d)|x_d|^{p-1} \right),
\end{align*}
hence the set of (second-order) well-positioned placements in $\ell_p^d$ will always be an open conull set.
It was shown in \cite{kit-pow-1} that for any $p \in (1,\infty)$,
the set of infinitesimally rigid placements of a given graph in a $\ell_p$ normed space will either be an open conull set or empty.
We shall now show that the three types of ``finite'' rigidity are equivalent in $\ell_p$ normed spaces with $p \notin \{1,\infty\}$.
We will first need the following technical result.

\begin{theorem}[{\cite[Thm.~5.4.8]{analytic}}]\label{t:uniform}
    Let $S \subset \mathbb{R}^n$ be a real analytic set (i.e.~the zero set of a real analytic function defined on an open subset of $\mathbb{R}^n$).
    Then there exists a real analytic manifold $M$ and a real analytic map $\phi : M \rightarrow \mathbb{R}^n$ with $\phi(M) = \overline{S}$,
    where $\overline{S}$ is the closure of $S$.
\end{theorem}

\begin{proposition}\label{p:lp}
    Let $(G,q)$ be a framework in an $\ell_p$ normed space for $p \in (1, \infty)$.
    Then the following are equivalent:
    \begin{enumerate}[(i)]
        \item $(G,q)$ is locally rigid.
        \item $(G,q)$ is continuously rigid.
    \end{enumerate}
\end{proposition}

\begin{proof}
    We begin with the following set up.
    Define the set $S := \{-1,0,1\}^d$ and function $s : \mathbb{R}^d \rightarrow S$ with $s(x) = (\sgn(x_i))_{i=1}^d$.
    Fix an ordering on the vertices of $V$ and define for every $h \in S^E$ the set 
    \begin{align*}
        C_h := \left\{ r \in (\mathbb{R}^d)^V : s(r_v-r_w) = h_{vw} \text{ for all $vw \in E$ with $v > w$} \right\}.
    \end{align*}
    We note that each set $C_h$ is a convex cone that is a open subset of its linear span,
    and the rigidity map $f_G$ is a real analytic map on the set $C_h$ when restricted to the linear span of $C_h$.
    Each set $f_G^{-1}f_G(q) \cap C_h$ is a real analytic subset of the linear space $\spann C_h$.
    
    Suppose that $(G,q)$ is locally flexible.
    Since the set $S^E$ is finite,
    it follows that there exists $h \in S^E$ such that for every $\varepsilon >0$ there exists $q' \in \overline{C_h}$ where $(G,q')$ and $(G,q)$ are equivalent but not congruent and $\|q'_v-q_v\|_p< \varepsilon$ for all $v \in V$.
    By \Cref{t:uniform},
    there exists a real analytic manifold $M_h$ and a real analytic map $\phi: M_h \rightarrow (\mathbb{R}^d)^V$ so that $\phi(M_h) = \overline{f_G^{-1}[f_G(q)] \cap C_h}$.
    Hence $\overline{f_G^{-1}f_G(q) \cap C_h}$ is locally connected,
    and we can choose an equivalent but not congruent placement $q'$ in the same connected component $C$ of $\overline{f_G^{-1}[f_G(q)] \cap C_h}$ as $q$.
    Since $C$ is connected, $\phi$ is continuous and $M_h$ is locally compact,
    there exists a connected component $C'$ of $M_h$ where $\phi(C') = C$.
    Choose a real analytic path $\alpha :[0,1] \rightarrow C'$ so that $\phi\circ \alpha(0) = q$ and $\phi \circ\alpha(1) = q'$.
    It is now immediate that $\gamma := \phi \circ \alpha$ is a non-trivial continuous flex of $(G,q)$.
\end{proof}

The conditions given for prestress stability and second-order rigidity can be significantly simplified for frameworks in $\ell_p$ normed spaces.
We define for every $x \in \mathbb{R}^d$ and $p \in (1, \infty)$ the vector $x^{(p-1)} := ( \sgn(x_i)|x_i|^{p-1})_{i=1}^d$, and for every suitable $p \in (1,\infty)$ we define $d \times d$ diagonal matrix $\Delta^{(p-2)}_x$ with diagonal $|x_1|^{p-2}, \ldots, |x_d|^{p-2}$.
If $p >2$ then $\Delta^{(p-2)}_x$ exists for all points $x$,
and if $p<2$ then $\Delta^{(p-2)}_x$ is only defined if $x$ has non-zero coordinates.
For the special case of $p=2$,
we can define $\Delta^{(0)}_x$ for all $x$ by setting it to always be the identity matrix.
By applying the methods used in \Cref{sec:2nd} with $f_G$ replaced with the map $q \mapsto (\|q_v-q_w\|_p^p)_{vw \in E}$,
we obtain the following equivalent characterisations of prestress stability and second-order rigidity in $\ell_p$ normed spaces.

\begin{corollary}\label{c:equiv defs}
    Let $(G,q)$ be a second-order well-positioned framework in an $\ell_p$ normed space for $p \in (1,\infty)$.
    Then the following properties hold.
    \begin{enumerate}[(i)]
        \item \label{c:equiv defs1} $(G,q)$ is prestress stable if and only if there exists $a \in \mathbb{R}^E$ so that for every $v \in V$ we have $\sum_{w \in N(v)} a_{vw} (q_v-q_w)^{(p-1)} = 0$, and
        for every edge $vw \in E$ we have
        \begin{align*}
            H^p_a(u) := \sum_{vw \in E} a_{vw} (u_v - u_w)^T \Delta^{(p-2)}_{q_v-q_w} (u_v - u_w) > 0
        \end{align*}
        for all $u \in \ker df_G(q) \setminus \mathcal{T}(q)$.
        \item \label{c:equiv defs2} $(G,q)$ is second-order rigid if and only if the following holds:
        if $(u,u') \in (\mathbb{R}^d)^V \times (\mathbb{R}^d)^V$ with $u \in \ker df_G(q)$ and
        \begin{align*}
            (p-1)(u_v-u_w)^T \Delta^{(p-2)}_{q_v-q_w} (u_v-u_w) + (u_v'-u_w')^T (q_v-q_w)^{(p-1)}  = 0
        \end{align*}
        for each $vw \in E$,
        then $u \in \mathcal{T}(q)$.
    \end{enumerate}
\end{corollary}

\section{Examples in various \texorpdfstring{$\ell_p$}{Lp} normed planes}\label{sec:ex}

For this section we shall focus on $\ell_p$ normed planes for $p \neq 2$.
In this class of normed planes we always have $\mathcal{T}(q)=\{(z)_{v\in V} : z\in \mathbb{R}^2\}$,
which will often simplify things.

\subsection{Examples for the \texorpdfstring{$\ell_\infty$}{L-infinity} normed plane}
We will begin with examples for the $\ell_\infty$ normed plane.
It follows from \Cref{p:poly} that our interesting examples shall be those where $(G,p)$ is not well-positioned.
We begin by revisiting the problematic framework depicted in \Cref{fig2}(ii).

\begin{example}[Revisiting \Cref{fig2}(ii)]\label{ex:1bar}
    Let $(G,p)$ be the framework in the $\ell_\infty$ normed plane pictured in \Cref{fig2}(ii);
    in particular,
    fix $G = ( \{a,b\},\{ab\})$ and $p_a=(0,0), p_b=(1,1)$.
    The set of generalised rigidity operators of $(G,p)$ is exactly the set $\{ R_t : t \in [0,1]\}$,
    where 
    \begin{align*}
        R_t : (\mathbb{R}^2)^{\{a,b\}} \rightarrow \mathbb{R}^{\{ab\}}, ~ u = \left(\left(u_a^1,u_a^2\right),\left(u_b^1,u_b^2\right)\right) \mapsto t\left(u_b^1- u_a^1\right) + (1-t)\left(u_b^2- u_a^2\right).
    \end{align*}
    By \Cref{p:altdef},
    $(G,p)$ is infinitesimally rigid.
    However $(G,p)$ is strongly infinitesimally flexible since $\rank R_t =1 < 2$ for each $t \in [0,1]$,
    which agrees with \Cref{mainthm}.
\end{example}

Strong infinitesimal rigidity importantly assumes that all generalised rigidity operators contain only trivial infinitesimal flexes in their kernels.
While it is tempting to weaken this condition to the requirement that this holds for a single generalised rigidity operator (a sufficient but not necessary condition for infinitesimal rigidity),
the following example tells us that we cannot do this and retain the useful equivalence properties that stem from \Cref{mainthm}.

\begin{example}[A strongly infinitesimally flexible framework with an ``infinitesimally rigid'' generalised rigidity operator]\label{ex1}
    Let $H$ be the graph described in \Cref{fig1}(i).
    Define the following placement $p$ of $H$ in the $\ell_\infty$ normed plane (as shown in \Cref{fig1}(ii)):
    \begin{gather*}
        p_{v_1} = (0,0), \qquad p_{v_2} = (1,0), \qquad  p_{v_3} = (0.9,1), \\ p_{v_4} =  (-0.1,1), \qquad p_{v_5} = p_{v_6} =  (0.5,1.2), \qquad p_{v_7} =  (1.3,2).
    \end{gather*}
    It is immediate that $(H,p)$ is locally and continuously flexible,
    as $p_{v_7}$ can swing freely around in a square path centered at $p_{v_5}=p_{v_6}$.
    Any generalised derivative of $(H,p)$ can be represented by the Jacobian matrix
    \begin{align*}
    \setcounter{MaxMatrixCols}{20}
        M(s,t) :=
        \scalemath{0.65}{
        \begin{bmatrix}
            % v_1,x & v_1,y & v_2,x & v_2,y & v_3,x & v_3,y & v_4,x & v_4,y & v_5,x & v_5,y & v_6,x & v_6,y & v_7,x & v_7,y \\
            0 & 0 & 0 & 0 & 0& 0 & 0& 0  & -s& -(1-s)& 0 & 0& s & (1-s)\\
            0 & 0 & 0 & 0 & 0& 0 & 0& 0 & 0 & 0 & -t& -(1-t)& t & (1-t)\\
            -1 & 0 & 1 & 0 & 0 & 0& 0 & 0& 0 & 0& 0 & 0& 0 & 0\\
            0 & -1 & 0 & 0 & 0 & 1 & 0 & 0& 0 & 0& 0 & 0& 0 & 0\\
            0 & -1 & 0 & 0 & 0 & 0 & 0 & 1& 0 & 0& 0 & 0& 0 & 0\\
            0 & 0 & 0 & 1 & 0 & -1& 0 & 0& 0 & 0& 0 & 0& 0 & 0\\
            0 & 0 & -1 & 0 & 0 & 0& 1 & 0& 0 & 0& 0 & 0& 0 & 0\\
            0 & 0 & 0 & 0 & 1& 0 & -1& 0  & 0& 0& 0 & 0& 0 & 0\\
            0 & -1 & 0 & 0 & 0 & 0 & 0 & 0& 0 & 1& 0 & 0& 0 & 0\\
            0 & 0 & 0 & -1 & 0 & 0 & 0& 0 & 0& 0 & 0& 1 & 0& 0\\
            0 & 0 & 0 & 0 & 1& 0 & 0& 0 & 0& 0 & -1& 0 & 0& 0 \\
            0 & 0 & 0 & 0 & 0& 0 & -1& 0  & 1& 0& 0 & 0& 0 & 0\\
        \end{bmatrix}
        }
    \end{align*}
    for $s,t \in [0,1]$.
    We observe that $\rank M(s,t) = 11$ if $s=t$,
    while $\rank M(s,t) = 12 = |E|$ if $s \neq t$.
    Hence although $(H,p)$ is strongly infinitesimally flexible,
    it has generalised rigidity operators with $\ker \delta f_G(p) = \mathcal{T}(p)$.
\end{example}

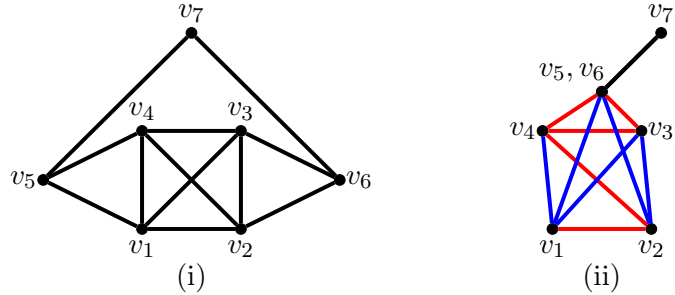
\begin{figure}[htp]
	\begin{center}
        \begin{tikzpicture}[scale=1.3]
			\node[vertex] (1) at (0,0) {};
			\node[vertex] (2) at (1,0) {};
			\node[vertex] (3) at (1,1) {};
			\node[vertex] (4) at (0,1) {};
			
			\node[vertex] (5) at (-1,0.5) {};
			\node[vertex] (6) at (2,0.5) {};
			\node[vertex] (7) at (0.5,2) {};
			
			\node at (0,-0.2) {$v_1$};
			\node at (1,-0.2) {$v_2$};
			\node at (1,1.2) {$v_3$};
			\node at (0,1.2) {$v_4$};
			
			\node at (-1.2,0.5) {$v_5$};
			\node at (2.2,0.5) {$v_6$};
			\node at (0.5,2.2) {$v_7$};
			
			\draw[edge] (1)edge(2);
			\draw[edge] (1)edge(3);
			\draw[edge] (2)edge(3);
			\draw[edge] (2)edge(4);
			\draw[edge] (3)edge(4);
			\draw[edge] (4)edge(1);
			
			\draw[edge] (1)edge(5);
			\draw[edge] (4)edge(5);
			
			\draw[edge] (2)edge(6);
			\draw[edge] (3)edge(6);
			
			\draw[edge] (5)edge(7);
			\draw[edge] (6)edge(7);
			
			\node at (0.5,-0.5) {(i)};
		\end{tikzpicture}\qquad\qquad
        \begin{tikzpicture}[scale=1.3]
			\node[vertex] (1) at (0,0) {};
			\node[vertex] (2) at (1,0) {};
			\node[vertex] (3) at (0.9,1) {};
			\node[vertex] (4) at (-0.1,1) {};
			
			\node[vertex] (5) at (0.5,1.4) {};
			\node[vertex] (6) at (0.5,1.4) {};
			\node[vertex] (7) at (1.1,2) {};
			
			\node at (0,-0.2) {$v_1$};
			\node at (1,-0.2) {$v_2$};
			\node at (1.1,1) {$v_3$};
			\node at (-0.3,1) {$v_4$};
			
			\node at (0.2,1.6) {$v_5,v_6$};
			\node at (1.1,2.2) {$v_7$};
			
			\draw[edge,red] (1)edge(2);
			\draw[edge,red] (2)edge(4);
			\draw[edge,red] (3)edge(4);
			\draw[edge,blue] (4)edge(1);
			\draw[edge,blue] (1)edge(3);
			\draw[edge,blue] (2)edge(3);
			
			\draw[edge,blue] (1)edge(5);
			\draw[edge,red] (4)edge(5);
			
			\draw[edge,blue] (2)edge(6);
			\draw[edge,red] (3)edge(6);
			
			\draw[edge] (5)edge(7);
			\draw[edge] (6)edge(7);
			
			\node at (0.5,-0.5) {(ii)};
		\end{tikzpicture}
	\end{center}
	\caption{(i) A graph $H$ that has rigid placements in the $\ell_\infty$ normed plane. (ii) A badly-positioned placement $p$ of $H$ in the $\ell_\infty$ normed plane as described in described in \Cref{ex1} that is locally and continuously flexible.
	An edge $v_iv_j$ is red edge (respectively, blue) if the Jacobian of the norm at $p_{v_i}-p_{v_j}$ is either $[1 ~ 0]$ or $[- 1 ~ 0]$ (respectively, either $[0 ~ 1]$ or $[0~ -1]$).}\label{fig1}
\end{figure}

As shown by \Cref{ex:1bar,ex1}, the equivalence stated in \Cref{p:poly}(\ref{p:poly3}) does not hold for badly-positioned frameworks.
One may hope that strong infinitesimal rigidity is still equivalent to local and continuous rigidity for the $\ell_\infty$ normed plane.
The next example indicates that this is unfortunately not true either.

\begin{example}[A strongly infinitesimally flexible framework that is locally and continuously rigid]\label{ex2}
    Let $G$ be the complete graph on the vertex set $\{v_1,v_2,v_3,v_4\}$ and let $p$ be the placement of $G$ with $p_{v_1}=(-1,-1)$, $p_{v_2}=(1,-1)$, $p_{v_3}=(-1,1)$, $p_{v_4}=(1,1)$.
    It was shown by Petty (see \cite[Thm.~4]{petty}) that every framework equivalent to $(G,p)$ is also congruent,
    hence $(G,p)$ is both locally and continuously rigid.
    
    For each $n \in \mathbb{N}$,
    define $(G,p^n)$ to be the well-positioned framework with $p^n_{v_1}=(-1,-1)$, $p^n_{v_2}=(1,-1)$, $p^n_{v_3}=(-1,1+\frac{1}{n})$, $p^n_{v_4}=(1,1+\frac{1}{n})$.
    Each rigidity operator $df_G(p^n)$ will be equal to the rank 5 linear map $T$ with the matrix representation
    \begin{align*}
    \setcounter{MaxMatrixCols}{20}
        \scalemath{0.75}{
        \begin{bmatrix}
            % v_1,x & v_1,y & v_2,x & v_2,y & v_3,x & v_3,y & v_4,x & v_4,y  \\
            -1 & 0 & 1 & 0 & 0& 0 & 0& 0\\
            0 & -1 & 0 & 0 & 0& 1 & 0& 0 \\
            0 & -1 & 0 & 0 & 0 & 0& 0& 1 \\
            0 & 0 & 0 & -1 & 0 & 1 & 0 & 0\\
            0 & 0 & 0 & -1 & 0 & 0 & 0 & 1 \\
            0 & 0 & 0 & 0 & -1 & 0 & 1 & 0 
        \end{bmatrix}.
        }
    \end{align*}
    Hence, $(G,p)$ is strongly infinitesimally flexible as $T \in \partial f_G(p)$.
\end{example}

\subsection{An example in smooth \texorpdfstring{$\ell_p$}{Lp} normed planes with colinear points}\label{sec:ex2}

For this section we shall define a continuous family of locally rigid frameworks which exhibits a variety rigidity properties.
To do so,
we will fix the following framework.

Fix $p>2$.
Let $(G,q)$ be any infinitesimally rigid framework in $\ell_p^2$ with no non-zero stresses and $q_v \neq q_w$ for all distinct vertices $v,w\in V$;
the existence of such frameworks can be guaranteed by \cite{kit-pow-1}.
Choose two vertices $v_1,v_2 \in V$.
Define $G'=(V',E')$ to be the graph formed by adding a vertex $v_0$ and the edges $v_0 v_1, v_0 v_2$.
We now define $q'$ to be the placement of $G'$ with $q'_v=q_v$ for all $v \in V$ and $q'_{v_0} =  (1-t) q_{v_1} + t q_{v_2}$ for some $t \in (0,1)$.
It is immediate that $(G',q')$ is locally rigid.
See \Cref{fig3} for an example of the framework $(G',q')$.

\begin{figure}[htp]
	\begin{center}
        \begin{tikzpicture}[scale=1.3]
			\node[vertex] (1) at (0,0) {};
			\node[vertex] (2) at (1,0.1) {};
			\node[vertex] (3) at (1.1,1.2) {};
			\node[vertex] (4) at (-0.1,1.1) {};
			
			\node[vertex] (5) at (-1,1.8) {};
			\node[vertex] (6) at (2,1.8) {};
			\node[vertex] (7) at (0.3,1.8) {};

			\node at (-1.2,1.8) {$v_1$};
			\node at (2.2,1.8) {$v_2$};
			\node at (0.3,2) {$v_0$};
			
			\draw[edge] (1)edge(2);
			\draw[edge] (1)edge(3);
			\draw[edge] (2)edge(3);
			\draw[edge] (2)edge(4);
			\draw[edge] (3)edge(4);
			\draw[edge] (4)edge(1);
			
			\draw[edge] (1)edge(5);
			\draw[edge] (4)edge(5);
			
			\draw[edge] (2)edge(6);
			\draw[edge] (3)edge(6);
			
			\draw[edge] (5)edge(7);
			\draw[edge] (6)edge(7);
		\end{tikzpicture}
	\end{center}
	\caption{An example of the framework $(G',q')$ described in \Cref{sec:ex2}.
	If the vertex $v_0$ is removed, then we obtain the infinitesimally rigid framework $(G,q)$.
	The point $q'_{v_0}$ is chosen to lie on the line between $q_{v_1}$ and $q_{v_2}$.}\label{fig3}
\end{figure}
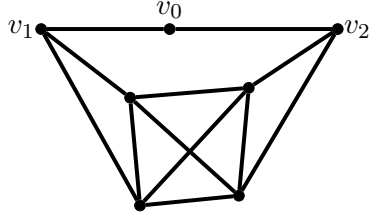

Let $a \in \mathbb{R}^{E'}$ be the unique (up to scalar) non-zero vector where $\sum_{w \in N(v)} a_{vw} (q'_v-q'_w)^{(p-1)} =0$ for all $v \in V'$.
As
\begin{align*}
   a_{v_0 v_1} (q'_{v_0}-q'_{v_1})^{(p-1)} + a_{v_0 v_2} (q'_{v_0}-q'_{v_2})^{(p-1)} = ( -t^{p-1} a_{v_0 v_1} + (1-t)^{p-1} a_{v_0 v_2} ) (q_{v_1}-q_{v_2})^{(p-1)},
\end{align*}
we have that $a_{v_0 v_1} = a_{v_0 v_2} (1-t)^{p-1}/t^{p-1}$.
If $a_{v_0 v_1} = a_{v_0 v_2} = 0$ then $a_e = 0$ for all $e \in E$ (as $(G,q)$ has no non-trivial stresses),
hence both $a_{v_0 v_1}$ and $a_{v_0 v_2}$ are non-zero.
Since the trivial infinitesimal flexes correspond with the translations,
we note that for any $u \in \ker df_{G'}(q')$, there exists $z \in \mathbb{R}^2$ so that $u_v = z$ for all $v \in V$,
and $u_{v_0} = y +z$ for some point $y$ orthogonal to $(q_{v_1}-q_{v_2})^{(p-1)}$ (with respect to the dot product).
We will now set $y =(y_1,y_2)$ and $q_{v_1} -q_{v_2} = ( q_1, q_2 )$ and observe the two different possibilities.

\begin{example}[$q_1,q_2 \neq 0$]\label{ex:prestress not inf}
    Given such an infinitesimal flex $u$, we note that
    \begin{eqnarray*}
        H^p_a(u) &=& a_{v_0 v_1} \left( y^T \Delta^{(p-2)}_{q'_{v_0} - q'_{v_1}} y \right) + a_{v_0 v_2} \left( y^T \Delta^{(p-2)}_{q'_{v_0} - q'_{v_2}} y \right) \\ 
        &=& \left(t^{p-2} a_{v_0 v_1} + (1-t)^{p-2}a_{v_0 v_2} \right)  y^T \Delta^{(p-2)}_{(q_1,q_2)} y. 
    \end{eqnarray*}
    Since
    \begin{align}\label{eq:prestress}
       y^T \Delta^{(p-2)}_{(q_1,q_2)} y = |q_1|^{q-2} y_1^2 + |q_2|^{q-2} y_2^2 >0,
    \end{align}
    it follows that $H_a(u) > 0$ for all $u \in \ker df_G(q)$ if and only if
     \begin{align*}
       t^{p-2} a_{v_0 v_1} + (1-t)^{p-2}a_{v_0 v_2} = (t^{-1}(1-t)^{p-1} + (1-t)^{p-2}) a_{v_0 v_2} >0.
    \end{align*}
    We can always choose whether $a_{v_0 v_2}$ is positive or negative by scaling $a$,
    hence $(G',q')$ is prestress stable by \Cref{c:equiv defs}(\ref{c:equiv defs1}).
\end{example}

% \begin{example}[$q_1,q_2 \neq 0$, $t = 1/2$]\label{ex:2nd not prestress}
%     As shown in \Cref{ex:prestress not inf},
%     $(G',q')$ is not prestress stable.
%     Choose $u' \in X^V$ so that
%     \begin{align*}
%         (p-1)(u_v-u_w)^T \Delta^{(p-2)}_{q'_v-q'_w} (u_v-u_w) + (u_v'-u_w')^T (q'_v-q'_w)^{(p-1)}  = 0
%     \end{align*}
%     for each $vw \in E'$.
%     We first note that there must exist $x' \in \mathbb{R}^2$ so that $u'_v = x'$ for all $v \in V$.
%     Set $u'_{v_0} = y' + x'$ for some $y' \in \mathbb{R}^2$.
%     By observing the second-order flex condition at $v_0v_1$ and $v_0 v_2$,
%     we have that
%     \begin{align*}
%         (p-1) t^{p-2} \left( y^T \Delta^{(p-2)}_{q_{v_1} - q_{v_2}} y \right) - t^{p-1} (y')^T (q'_{v_1} - q'_{v_2}) = 0, \\
%         (p-1) (1-t)^{p-2} \left( y^T \Delta^{(p-2)}_{q_{v_1} - q_{v_2}} y \right) + (1-t)^{p-1} (y')^T (q'_{v_1} - q'_{v_2}) = 0.
%     \end{align*}
%     By substituting $t=1/2$ and dividing through by $(1/2)^{p-2}$, we obtain the equations
%     \begin{align*}
%         (p-1) \left( y^T \Delta^{(p-2)}_{q_{v_1} - q_{v_2}} y \right) - \frac{1}{2} (y')^T (q'_{v_1} - q'_{v_2}) = 0, \\
%         (p-1) \left( y^T \Delta^{(p-2)}_{q_{v_1} - q_{v_2}} y \right) + \frac{1}{2} (y')^T (q'_{v_1} - q'_{v_2}) = 0.
%     \end{align*}
%     As $y^T \Delta^{(p-2)}_{q_{v_1} - q_{v_2}} y >0$,
%     it follows that $(u,u')$ is not a second-order flex,
%     and hence $(G',q')$ is second-order rigid.
% \end{example}

\begin{example}[$q_2 = 0$]\label{ex:local not 2nd}
    Since $q_2=0$,
    it follows from our choice of $y$ that $y_1=0$.
    With this, \cref{eq:prestress} becomes an equality.
    Set $u' \in (\mathbb{R}^2)^{V'}$ with $u'_v = 0$ for all $v \in V$ and $u'_{v_0} = u_{v_0}$.
    It follows that $(u,u')$ is a second-order flex of $(G',q')$, hence $(G',q')$ is second-order flexible by \Cref{c:equiv defs}(\ref{c:equiv defs2}).
\end{example}

\subsection{Braced grids in the $\ell_p$ plane}

Fix $p >2$.
Let $K_4$ be the complete graph with vertex set $V = \{v_1,v_2,v_3,v_4\}$,
and let $q$ be the placement of $K_4$ in the $\ell_p$ normed plane where
$q_{v_1}=(0,0)$, $q_{v_2}=(1,0)$, $q_{v_3}=(0,1)$, $q_{v_4}=(1,1)$.
The framework $(G,q)$ is infinitesimally flexible as it has a non-trivial infinitesimal flex
$\bar{u}$ with $\bar{u}_{v_1}=(1,-1)$, $\bar{u}_{v_2}=(1,1)$, $\bar{u}_{v_3}=(-1,-1)$, $\bar{u}_{v_4}=(-1,1)$.

Let $a \in \mathbb{R}^E$ be the element where $a_{v_1 v_2} = a_{v_1 v_3} = a_{v_2 v_4}= a_{v_3 v_4} = -1$
and $a_{v_1 v_4}=a_{v_2 v_3} = 1$.
We compute that $\sum_{j \neq i} a_{v_i v_j}(q_{v_i}-q_{v_j})^{p-1} = 0$ for each $i\in \{1,2,3,4\}$.
We can also easily compute the Hessians of the various edge constraints to be
\begin{align*}
    \Delta^{(p-2)}_{q_{v_1} - q_{v_2}} = \Delta^{(p-2)}_{q_{v_2} - q_{v_1}} = \Delta^{(p-2)}_{q_{v_3} - q_{v_4}} = \Delta^{(p-2)}_{q_{v_4} - q_{v_3}} = 
    \begin{bmatrix}
        1 & 0 \\
        0 & 0
    \end{bmatrix},\\
    \Delta^{(p-2)}_{q_{v_1} - q_{v_3}} = \Delta^{(p-2)}_{q_{v_3} - q_{v_1}} = \Delta^{(p-2)}_{q_{v_2} - q_{v_4}} = \Delta^{(p-2)}_{q_{v_4} - q_{v_2}} = 
    \begin{bmatrix}
        0 & 0 \\
        0 & 1
    \end{bmatrix},\\
    \Delta^{(p-2)}_{q_{v_1} - q_{v_4}} = \Delta^{(p-2)}_{q_{v_4} - q_{v_1}} = \Delta^{(p-2)}_{q_{v_2} - q_{v_3}} = \Delta^{(p-2)}_{q_{v_3} - q_{v_2}} = 
    \begin{bmatrix}
        1 & 0 \\
        0 & 1
    \end{bmatrix}.
\end{align*}
Choose any infinitesimal flex $u \in \ker df_{K_4}(q)$.
It can quickly checked that $u = ( c \bar{u}_{v_i} + x)_{i=1}^4$,
where $x \in \mathbb{R}^2$ and $c \in \mathbb{R}$.
With this we compute that $H_a^p(u) = c^2 H_a^p(\bar{u}) = 4 c^2$,
hence $H_a^p$ takes only positive values on the set of non-trivial infinitesimal flexes.
By \Cref{c:equiv defs}(\ref{c:equiv defs1}) and \Cref{main2nd},
$(G',q')$ is prestress stable, locally rigid and continuously rigid.

We can now extend the properties of this framework to what we shall define to be a \emph{doubly-braced grid}.
Given $[k]$ is the first $k$ positive integers,
let $H$ be the graph with vertex set $[m] \times [n]$ and an edge between two vertices $(i,j),(i',j')$ if and only if $i=i'$ and $|j-j'|=1$, or $|i-i'|=1$ and $j=j'$.
Now define a set $B$ of double-braces of the grid $H$,
i.e.~pairs of edges $b_{i,j} = \{(i,j)(i+1,j+1), \, (i+1,j)(i,j+1)\}$.
Let $G=(V,E)$ be the graph formed from $H$ by adding the double-braces to $H$ as pairs of edges,
and fix $q:V \rightarrow \mathbb{R}^2$ to be the placement of $G$ with $p_{(i,j)} = (i,j)$.
See \Cref{fig:grid} for two examples of such grids.
With this, the following holds.

\begin{proposition}\label{p:grid}
    A doubly-braced grid $(G,q)$ is prestress stable/locally rigid/continuously rigid in the $\ell_p$ normed plane for $p>2$ if and only if for every $i \in [m-1]$ and $j \in [n-1]$,
    there exists double-braces $b_{i,k}, b_{\ell,j} \in B$.
\end{proposition}

\begin{proof}
    If a column/row does not contain any braces, then it is easy to see that the grid must be locally and continuously flexible (see \Cref{fig:grid}(i)), and hence not prestress stable by \Cref{main2nd}.
    Suppose every row and column contains a double-brace.
    Fix the vector $a \in \mathbb{R}^E$ which takes the value $-1$ at every edge of a 4-cycle that has a double-brace, 1 at every double-brace edge and 0 for all other edges.
    Choose a non-trivial infinitesimal flex $u$ of $(G,q)$.
    We note an important property of $u$:
    there exists $x_1,\ldots,x_m \in \mathbb{R}$ and $y_1,\ldots,y_n \in \mathbb{R}$ so that $u_{(i,j)} = (x_i,y_i)$ for every $(i,j) \in [m] \times [n]$.
    As $u$ is non-trivial, either there exists $x_i \neq x_{i+1}$ or $y_j \neq y_{j+1}$ for some $i \in [m-1]$ or $j \in [n-1]$.
    Without loss of generality, we will suppose that $x_i \neq x_{i+1}$,
    since we may freely rotate $(G,q)$ by $90^\circ$ increments.
    Because every row contains a double-brace,
    there exists a $K_4$ subframework of $(G,q)$ on the vertices $(i,k),(i,k+1),(i+1,k),(i+1,k+1)$ for some $k \in [n-1]$ so that the restriction of $u$ is a non-trivial infinitesimal flex.
    As $u$ is non-trivial when restricted to a $K_4$ subframework of $(G,q)$,
    we have $H_a^p(u) = \sum_{vw \in \bigcup B}  \|u_v - u_w\|_2^2 > 0$.
    The result now follows from \Cref{c:equiv defs}(\ref{c:equiv defs1}) and \Cref{main2nd}.
\end{proof}

\begin{figure}[htp]
	\begin{center}
        \begin{tikzpicture}[scale=1]
			\node[vertex] (11) at (1,1) {};
			\node[vertex] (21) at (2,1) {};
			\node[vertex] (31) at (3,1) {};
			\node[vertex] (41) at (4,1) {};
			
			\node[vertex] (12) at (1,2) {};
			\node[vertex] (22) at (2,2) {};
			\node[vertex] (32) at (3,2) {};
			\node[vertex] (42) at (4,2) {};
			
			\node[vertex] (13) at (1,3) {};
			\node[vertex] (23) at (2,3) {};
			\node[vertex] (33) at (3,3) {};
			\node[vertex] (43) at (4,3) {};
			
			\node[vertex] (14) at (1,4) {};
			\node[vertex] (24) at (2,4) {};
			\node[vertex] (34) at (3,4) {};
			\node[vertex] (44) at (4,4) {};
			
			\node at (2.5,0.5) {(i)};
			
			\draw[edge] (11)edge(21);
			\draw[edge] (21)edge(31);
			\draw[edge] (31)edge(41);
			
			\draw[edge] (12)edge(22);
			\draw[edge] (22)edge(32);
			\draw[edge] (32)edge(42);
			
			\draw[edge] (13)edge(23);
			\draw[edge] (23)edge(33);
			\draw[edge] (33)edge(43);
			
			\draw[edge] (14)edge(24);
			\draw[edge] (24)edge(34);
			\draw[edge] (34)edge(44);
			
			\draw[edge] (11)edge(12);
			\draw[red,edge] (12)edge(13);
			\draw[edge] (13)edge(14);
			
			\draw[edge] (21)edge(22);
			\draw[red,edge] (22)edge(23);
			\draw[edge] (23)edge(24);
			
			\draw[edge] (31)edge(32);
			\draw[red,edge] (32)edge(33);
			\draw[edge] (33)edge(34);
			
			\draw[edge] (41)edge(42);
			\draw[red,edge] (42)edge(43);
			\draw[edge] (43)edge(44);
			
			\draw[edge] (11)edge(22);
			\draw[edge] (12)edge(21);
			
			\draw[edge] (31)edge(42);
			\draw[edge] (41)edge(32);

			\draw[edge] (13)edge(24);
			\draw[edge] (14)edge(23);
			
			\draw[edge] (33)edge(44);
			\draw[edge] (34)edge(43);
			
		\end{tikzpicture}\qquad\qquad
        \begin{tikzpicture}[scale=1]
			\node[vertex] (11) at (1,1) {};
			\node[vertex] (21) at (2,1) {};
			\node[vertex] (31) at (3,1) {};
			\node[vertex] (41) at (4,1) {};
			
			\node[vertex] (12) at (1,2) {};
			\node[vertex] (22) at (2,2) {};
			\node[vertex] (32) at (3,2) {};
			\node[vertex] (42) at (4,2) {};
			
			\node[vertex] (13) at (1,3) {};
			\node[vertex] (23) at (2,3) {};
			\node[vertex] (33) at (3,3) {};
			\node[vertex] (43) at (4,3) {};
			
			\node[vertex] (14) at (1,4) {};
			\node[vertex] (24) at (2,4) {};
			\node[vertex] (34) at (3,4) {};
			\node[vertex] (44) at (4,4) {};
			
			\node at (2.5,0.5) {(ii)};
			
			\draw[edge] (11)edge(21);
			\draw[edge] (21)edge(31);
			\draw[edge] (31)edge(41);
			
			\draw[edge] (12)edge(22);
			\draw[edge] (22)edge(32);
			\draw[edge] (32)edge(42);
			
			\draw[edge] (13)edge(23);
			\draw[edge] (23)edge(33);
			\draw[edge] (33)edge(43);
			
			\draw[edge] (14)edge(24);
			\draw[edge] (24)edge(34);
			\draw[edge] (34)edge(44);
			
			\draw[edge] (11)edge(12);
			\draw[edge] (12)edge(13);
			\draw[edge] (13)edge(14);
			
			\draw[edge] (21)edge(22);
			\draw[edge] (22)edge(23);
			\draw[edge] (23)edge(24);
			
			\draw[edge] (31)edge(32);
			\draw[edge] (32)edge(33);
			\draw[edge] (33)edge(34);
			
			\draw[edge] (41)edge(42);
			\draw[edge] (42)edge(43);
			\draw[edge] (43)edge(44);

			\draw[edge] (13)edge(24);
			\draw[edge] (14)edge(23);
			
			\draw[edge] (22)edge(33);
			\draw[edge] (23)edge(32);
			
			\draw[edge] (31)edge(42);
			\draw[edge] (41)edge(32);
			
		\end{tikzpicture}
	\end{center}
	\caption{(i) A continuously flexible doubly-braced grid in the $\ell_p$ normed plane. A non-trivial continuous flex can be formed by rotating the red edges in unison whilst maintaing the orientation of all other edges.
	(ii) A prestress stable doubly-braced grid in the $\ell_p$ normed plane for $p>2$.
	The grid is not prestress stable for $1 \leq p < 2$ since it is not second-order well-positioned, and is in fact not even continuously rigid for $p=2$.}\label{fig:grid}
\end{figure}
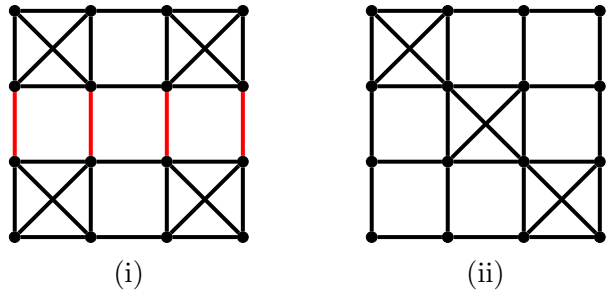

\begin{remark}
    Infinitesimal rigidity properties of doubly-braced grids in $\ell_p$ normed planes were very recently investigated by Power \cite{power20}.
    While his methods allow for a more general class of grids in a larger class of normed planes,
    they do require that each doubly-braced grid square be infinitesimally rigid.
    As was previously shown, this will not be true for the doubly-braced grids described in \Cref{p:grid}, and so his methods cannot be used here.
\end{remark}

\section{Closing remarks}\label{sec:end}

In \Cref{sec:summary},
we described the set of implications we have proven throughout the paper.
If we look at the first set of implications for frameworks that may or may not be well-positioned, we can see that no other possible implications will hold for all normed spaces:
\Cref{fig2}(i) is (strongly) infinitesimally flexible but locally and continuously rigid;
\Cref{fig2}(ii) is infinitesimally rigid but locally and continuously flexible;
and \cite{D21} details a locally flexible but continuously rigid framework.
Most of the remaining possible implications for second-order well-positioned frameworks are also false:
\Cref{ex:prestress not inf} is (strongly) infinitesimally flexible but prestress stable;
\cite[Fig.~9(b)]{conn96} is not prestress stable but is second-order rigid;
and \Cref{ex:local not 2nd} is not prestress stable or second-order rigid but is locally and continuously rigid.
It is, however, unknown if second-order rigidity implies local or continuous rigidity.
The method Connelly and Whiteley employ to show that second-order rigidity implies local and continuous rigidity in Euclidean normed spaces heavily relies on both the norm being infinitely differentiable and local flexibility implying the existence of an analytic flex.
Since there exist norms with no thrice-differentiable points\footnote{Take a positive continuous real-valued function $f:[a,b] \rightarrow \mathbb{R}$ (for example, a section of the Weierstrass function) and define the twice-differentiable (but not thrice-differentiable) concave function $F:[a,b] \rightarrow \mathbb{R}$ by setting $F(x) := -\int_a^x \int_a^t f(s) \mathop{ds} \mathop{dt}$. By applying the method outlined in \cite{govc} to $F$ (possibly shifted with some alteration of its domain), we can obtain a norm with no third-order derivative anywhere.}, we conjecture the following.

\begin{conjecture}\label{conj1}
    There exists a normed space with an second-order rigid framework that is locally/continuously flexible.
\end{conjecture}

Some evidence to back \Cref{conj1} stems from that fact that even in Euclidean spaces,
third-order rigidity does not imply local or continuous rigidity \cite{connserv}.
This is due to a key property that holds for second-order rigidity but not for third-order rigidity;
namely, if $(u^{1},\ldots,u^{2n})$ is a $2n^{\text{th}}$-order infinitesimal flex of a second-order rigid framework in a Euclidean space, then there exists a rigid body motion $\gamma$ of the framework with $\gamma^{(i)}(0) = u^i$ for each $i \in \{1,\ldots,n\}$ \cite{conn80}.
Since this property cannot be used in general normed spaces due to the lack of smoothness of general norms.
Combined with the fact that being rigid with respect to analytic flexes does not imply local or continuous rigidity in general normed spaces (for example, see \Cref{fig2}(ii)),
it seems logical that \Cref{conj1} should be true.

\subsection*{Acknowledgements}
The author would like to thank Derek Kitson for his helpful advice and suggestions for the paper.

\end{document}